\documentclass[12pt]{amsart}

\usepackage{amsmath,amscd,amsthm,amssymb,amsxtra,latexsym,epsfig,epic,graphics,mathdots,amssymb, enumerate}
\usepackage[all]{xy}
\SelectTips{cm}{11}
\usepackage[utf8]{inputenc}
\usepackage[svgnames]{xcolor}
\usepackage{tcolorbox}
\tcbuselibrary{skins,listings,theorems}
\usepackage{hyperref}

\usepackage[all]{xy}

\newtheorem{thm}{Theorem}

\newtheorem{cor}[thm]{Corollary}

\theoremstyle{definition}

\newtheorem{exmp}[thm]{Example}

\newtheorem{rem}[thm]{Remark}          

\newtheorem*{ack}{Acknowledgments}      

\newtheorem{defn-thm}[thm]{Definition--Theorem}  
\newtheorem{defn-lem}[thm]{Definition--Lemma}  

\theoremstyle{remark}

\newcommand{\z}[0]{{\mathbb Z}}

\newcommand{\p}[0]{{\mathbb P}}

\newcommand{\ext}[0]{\operatorname{Ext}}    
\newcommand{\Hom}[0]{\operatorname{Hom}}

\def\loccoh#1.#2.#3.#4.{H^{#1}_{#2}(#3,#4)}

\DeclareMathAlphabet{\mathchanc}{OT1}{pzc}%
                                {m}{it}

\begin{document}
\bibliographystyle{amsalpha}


\title[Syzygy structure theorem for del Pezzo varieties]{A structure theorem for syzygies of del Pezzo varieties}
\author{Yeongrak Kim}
\email{yeongrak.kim@pusan.ac.kr}

\address{Department of Mathematics \& Institute of Mathematical Science,  Pusan National University, 2 Busandaehak-ro 63beon-gil, Geumjeong-gu, 46241 Busan, Korea}

\begin{abstract}
Using the Buchsbaum-Eisenbud structure theorem for a minimal free resolution of an arithmetically Gorenstein variety, we describe a structure theorem for the highest linear syzygies among quadrics defining a del Pezzo variety. Indeed, such syzygies can be represented as columns of a skew-symmetric matrix whose entries are wedge products of linear forms.
\end{abstract}


\maketitle

\section{Introduction}\label{Sect:Introduction}
Let $X \subset \p^r$ be a smooth, non-degenerate projective variety of dimension $n$, codimension $e=r-n$, degree $d = [\mathcal O_X(1)]^n$ over an algebraicallay closed field $k=\overline{k}$ of characteristic $0$. We say that $X$ is a \emph{del Pezzo variety} if $\omega_X = \mathcal O_X (1-n)$. Such varieties were intensively studied from various viewpoints. For instance, a del Pezzo variety satisfies $d = e + 2$, so it belongs to the next case of varieties of minimal degree. In 1889, Castelnuovo intiated the study of quadrics defining $X$ so that $\dim_{k} (I_X)_2$ is at most $\binom{e+1}{2}$. In 1894, Fano noticed that a del Pezzo variety $X$ is defined by $\dim_{k} (I_X)_2 = \binom{e + 1}{2} - 1$ independent quadrics, so the homogeneous ideal of $X$ contains the next maximal number of quadrics that a variety. Both points of view tell us that del Pezzo varieties are the next simplest case to the varieties of minimal degrees, namely, hyperquadrics, rational normal scrolls, and the Veronese surface in the smooth case. It is also remarkable that del Pezzo varieties also appear as varieties having $(e-1)$-th linear syzygies among quadrics which is a part of the $K_{p,1}$-theorem of Green \cite[Theorem 3.c.1]{Gre84}. It is clear that del Pezzo varieties of dimension $1$ are elliptic normal curves $C \subset \p^r$ of degree $r+1$, so there are del Pezzo varieties of arbitrary degree $d \ge 3$. On the other hand, there are only finitely many families of del Pezzo varieties when $n \ge 2$, see the classification in \cite{Fuj90}. In particular, the degree of a del Pezzo variety satisfies $3 \le d \le 9$, and even the dimension of del Pezzo varieties is at most $6$ when $d \ge 5$. 

It is well-known that a del Pezzo variety is arithmetically Gorenstein, that is, the minimal free resolution of the homogeneous coordinate ring $S_X$ of $X$ is isomorphic to its dual resolution. In particular, the graded Betti numbers of $S_X$ are symmetric. We are particularly interested in linear syzygies of del Pezzo varieties. Using the Koszul cohomology groups, we may interpret a $p$-th syzygy among the quadrics defining $X$ as a cycle in $K_{p,1} (S_X, S_1)$ which is the cohomology group of the Koszul complex
\[
[\wedge^{p+1} S_1 \otimes (S_X)_{0} \to \wedge^p S_1 \otimes (S_X)_{1} \to \wedge^{p-1} S_1 \otimes (S_X)_{2}].
\]
When $X \subset \p^r$ is a non-degenerate projective variety so that the ideal $I_X$ of $X$ does not contain any linear form, then we have $K_{p,1} (S_X, S_1) \simeq K_{p-1,2} (I_X, S_1)$ for every $p \ge 1$, so we will treat them as identical objects throughout the paper.

The goal of this paper is to establish a structure theorem for the highest syzygies of a del Pezzo variety so that they can be expressed as a skew-symmetric matrix by choosing suitable bases, see Theorem \ref{thm:main}. The key idea is to express linear (and higher linear) syzygies among the generators of minimal degrees as the wedge product of matrices representing linear syzygies. Note that the case $e=3$ is particularly well-understood as the famous structure theorem of Buchsbaum and Eisenbud \cite{BE77}, and the case $e=4$ is also easily deduced from \cite{Rei15}. As an application, we show that the ideal defining a quadric syzygy scheme of a highest linear syzygy of a del Pezzo variety $X$ is strictly contained in the ideal $I_X$. 

\begin{ack}
The author thanks Frank-Olaf Schreyer for helpful discussion, and also thanks the anonymous referee for useful comments. This work was supported by a $2$-Year Research Grant of Pusan National University.
\end{ack}

\section{Gorenstein resolution and Syzygies}
%
 
We briefly recall the symmetrizer trick of Buchsbaum and Eisenbud \cite{BE77}. Let $X \subset \p^r$ be an arithmetically Gorenstein variety of dimension $n$ and codimension $e=r-n$, that is, the homogeneous coordinate ring $S_X = S/I_X$ has a dualizing module $\omega_X$ satisfying 
\[
\omega_X = \ext_S^{e} (S_X, S(-r-1)) \simeq S_X (t)
\]
for some integer twist $t \in \z$, where $S = k[x_0, \cdots, x_r]$ is the standard graded polynomial ring which is the homogeneous coordinate ring of $\p^r$, and $I_X \subset S$ is the ideal of $X$. Let $F_{\bullet}$ be a minimal free resolution of $S_X$. By passing to the total complex of $F_{\bullet} \otimes F_{\bullet}$, the symmetric square $Sym^2 (F_{\bullet})$ is a complex of graded free $S$-modules such that
\[
Sym^2 (F_{\bullet})_n = \left\{ \begin{array}{ll}
\bigoplus_{i+j=n, i<j} F_i \otimes F_j & \text{ if $n$ is odd,} \\
{ \left( \bigoplus_{i+j=n, i<j} F_i \otimes F_j  \right) } \oplus \bigwedge^2 (F_{n/2}) & \text{ if $n \equiv 2 \pmod 4$,} \\
{ \left( \bigoplus_{i+j=n, i<j} F_i \otimes F_j \right) } \oplus Sym^2 (F_{n/2}) & \text{ if $n \equiv 0 \pmod 4$.}
\end{array} \right.
\]
By considering the twisted dual resolution $F_{\bullet}^{\vee} = \Hom_S (F_{\bullet}, S(-r-t-1))$, we may lift the identity map $F_{0}^{\vee} \stackrel{\simeq} \longrightarrow F_{0}$ so that we have isomorphisms $F_i \simeq F_{e-i}^{\vee}$ for every $0 \le i \le e$. Indeed, we have a self-dual resolution $F_{\bullet} \simeq F_{\bullet}^{\vee}$ induced by the symmetrizer $Sym^2 (F_{\bullet}) \to F_{e} \simeq S(-r-t-1)$.

When $e = 2m+1$ is odd, then this is the case of \cite[Theorem 1.5, 4.2]{Ste22}. When $e=2m$ is even, then there is a non-degenerate symmetric (if $m$ is even), or a skew-symmetric (if $m$ is odd) bilinear form $\varphi : F_{m} \otimes F_{m} \to F_{e}$ compatible with the matrix $F_{\bullet}$, so that $d_i : F_{i} \to F_{i-1}$ is dual to $d_{e-i+1} : F_{e-i+1} \to F_{e-i}$ under $\varphi$, see for instance \cite[Section 2.1]{Rei15}.

Note that this (skew-)symmetricity is compatible with syzygies. Using the notion of Koszul cohomology groups, $p$-th syzygies of equations defining $X$ are encoded in $K_{p,q} (S_X, S_1)$. When $e=2m$ is even, then we have the multiplication map on $m$-th syzygies
\[
K_{m,q} (S_X, S_1) \otimes K_{m,r+t+1-q} (S_X, S_1) \to K_{2m, r+t+1} (S_X, S_1) \simeq k.
\]
For Koszul cycles $\alpha \in K_{m,q} (S_X, S_1)$ and $\beta \in K_{m,r+t+1-q} (S_X, S_1)$, we take their representatives in $\wedge^m S_1 \otimes S_{q}$ and $\wedge^m S_1 \otimes S_{r+t+1-q}$ and multiply them in a natural way. As a result, we obtain an element in $\wedge^{2m} S_1 \otimes S_{r+t+1}$ whose image under the Koszul differential map is zero, and hence, $\alpha \beta$ defines a Koszul cycle in $K_{2m, r+t+1} (S_X, S_1)$. Furthermore, we have $\alpha \beta = (-1)^{m^2} \beta \alpha$ which implies that the above multiplication map explains a $\pm$-symmetric perfect pairing $F_{m} \otimes F_m \to F_{2m} = S (-r-t-1)$ \cite[Section 1]{BE77}, which is symmetric when $e \equiv 0 \pmod 4$, and is skew-symmetric when $e \equiv 2 \pmod 4$.

For instance, let $S=k[t_1, \cdots, t_{2m}]$ be a standard graded polynomial ring, and consider the Koszul resolution $K_{\bullet}$ of $t_1, \cdots, t_{2m}$ which is the minimal free resolution of $k = S/(t_1, \cdots, t_{2m})$. The Koszul cohomology group $K_{m,0} (k,S_1)$ is a $k$-vector space of dimension $\binom{2m}{m}$ generated by the cycles
\[
[ t_{i_1} \wedge \cdots \wedge t_{i_m} \otimes 1]
\]
where $1 \le i_1 < \cdots < i_m \le 2m$. It is clear that the multiplication map $K_{m,0} (k, S_1) \otimes K_{m,0} (k, S_1) \to K_{2m,0} (k, S_1)$ is either symmetric (when $m$ is even) or skew-symmetric (when $m$ is odd), for instance,
\[
[t_1 \wedge \cdots \wedge t_m \otimes 1] \cdot [t_{m+1} \wedge \cdots \wedge t_{2m} \otimes 1] = [t_1 \wedge \cdots \wedge t_{2m} \otimes 1]
\]
whereas
\[
[t_{m+1} \wedge \cdots \wedge t_{2m} \otimes 1] \cdot [t_1 \wedge \cdots \wedge t_m \otimes 1] = (-1)^{m^2} [t_1 \wedge \cdots \wedge t_{2m} \otimes 1].
\]
By rearranging these basis vectors if necessary, we may write it as a non-degenerate bilinear form $\varphi : K_{m,0} (k, S_1) \otimes K_{m,0} (k, S_1) \to K_{2m,0} (k, S_1)$, and we may represent $\varphi$ as a block matrix $\varphi= \left( \begin{array}{c|c} 0 & I \\ \hline I & 0 \end{array} \right)$ when $m$ is even, and $\varphi= \left( \begin{array}{c|c} 0 & I \\ \hline -I & 0 \end{array}\right) $ when $m$ is odd.

We first discuss an expression of linear syzygies using the exterior product. Let $X \subset \p^r$ be a non-degenerate projective variety, and let $I_X$ be its homogeneous ideal generated by forms of degree $\ge d$ so that $(I_X)_q = 0$ for $q < d$. We take a linear strand consisting of linear and higher linear syzygies of equations in $(I_X)_d$ from the minimal free resolution of $X$, namely, a subcomplex
\[
0 \to S(-d-s)^{\beta_{s+1,d-1}} \stackrel{d_{s+1}} \to \cdots \stackrel{d_3} \to S(-d-1)^{\beta_{2,d-1}}  \stackrel{d_2} \to S(-d)^{\beta_{1,d-1}} \stackrel{d_1} \to S \to 0
\]
consisted of equations of degree $d$ in $I_X$ and their (higher) linear syzygies. Since $K_{p,q} (S, S_1) = 0$ for all $p>0$, it is clear that $K_{p,q} (S_X, S_1) \simeq K_{p-1, q+1} (I_X, S_1)$ for every $p \ge 1$ and $q < d$. In particular, the graded Betti number $\beta_{p, d-1}$ equals to $\dim_k K_{p,d-1} (S_X, S_1) = \dim_k K_{p-1, d} (I_X, S_1)$ in the above subcomplex.

We are interested in linear syzygies among degree $d$ equations of $X$ presented by the matrices $d_p$'s. We consider the composition of two consecutive syzygy matrices in a forceful way, indeed, each entry of the composed matrix is computed as the sum of tensors
\[
\begin{pmatrix} a_1 & \cdots & a_s \end{pmatrix}  \begin{pmatrix} b_1 \\ \vdots \\ b_s \end{pmatrix} = \sum_{i=1}^s (a_i \otimes b_i), \quad a_i, b_j \in S_1.
\]
Since the ``symmetric composition'' of two consecutive linear syzygy matrices $d_i d_{i+1} = 0$ for every $i$, we may regard each entry as an alternating tensor $\sum_{i=1}^s ( a_i \wedge b_i )$. By repeating the process, the wedge product $D_p = d_2 \wedge d_3 \wedge \cdots \wedge d_p$ where $2 \le p \le s+1$ can be regarded as the ``exterior composition'' of linear syzygy matrices up to degree $p$. In particular, $D_p$ is a $(\beta_{1,d-1} \times \beta_{s+1,d-1})$ matrix with entries in $\wedge^{p-1} S_1$. If we choose a set of degree $d$ generators $\{f_1, \cdots, f_{\beta_{1,d-1}}\}$ in $(I_X)_d$ consisting $d_1$, then each column of $D_p$ defines a $p$-th linear syzygy of the form $\gamma = \sum_{i=1}^{\beta_{1,d-1}} (\ell_{J_i} \otimes f_i) \in \wedge^{p-1} S_1 \otimes (I_X)_d$. It is clear that the Koszul differential $\delta : \wedge^{p-1} S_1 \otimes (I_X)_d \to \wedge^{p-2} S_1 \otimes (I_X)_{d+1}$ maps $\gamma$ to $0$, and hence $\gamma \in K_{p-1,d} (I_X, S_1) = \ker [ \wedge^{p-1} S_1 \otimes (I_X)_d \to \wedge^{p-2} S_1 \otimes (I_X)_{d+1}]$ corresponds to a cycle in $K_{p-1,d} (I_X, S_1) \simeq K_{p,d-1} (S_X, S_1)$. Indeed, the space of $p$-th linear syzygies among degree $d$ equations of $X$ coincides with the column space of $D_p$. Note that such a wedge product of differentials in a complex can be used to define the Atiyah class, see \cite[Section 1.2]{ALJ89}. It can be also helpful in explicit computations of resultants and Cayley-Chow forms via B{\'e}zout or Steifel expression, see \cite[Chapter 12, 13]{GKZ94}, \cite{ES03} and references therein.

\begin{exmp}
\begin{enumerate}[(1)]
\item Let $X \subset \p^4$ be the rational normal curve defined as the $2$-minors of the $1$-generic matrix $M = \begin{pmatrix} x_0 & x_1 & x_2 & x_3 \\ x_1 & x_2&x_3 & x_4 \end{pmatrix}$. The ideal of $X$ is generated by $6$ quadrics, namely, 
\[
\left\{
\begin{array}{rcl}
Q_0 & = & x_1^2 - x_0 x_2 \\
Q_1 & = & x_1 x_2 - x_0 x_3 \\
Q_2 & = & x_2^2 - x_1 x_3 \\
Q_3 & = & x_1 x_3 - x_0 x_4 \\
Q_4 & = & x_2 x_3 - x_1 x_4 \\
Q_5 & = & x_3^2 - x_2 x_4.
\end{array}
\right.
\]
The linear syzygies among $Q_0, \cdots, Q_5$ is given by matrices
\[
d_2 = \begin{pmatrix} -x_2 & x_3 & -x_3 & \cdot & x_4 & \cdot & \cdot & \cdot \\
x_1 & -x_2 & \cdot & -x_3 & \cdot & x_4 & x_4 & \cdot \\
-x_0 & x_1 & \cdot & \cdot & \cdot & -x_3 & \cdot & x_4 \\
\cdot & \cdot & x_1 & x_2 & -x_2 & -x_3 & -x_3 & \cdot \\
\cdot & \cdot & -x_0 & \cdot & x_1 & x_2 & \cdot & -x_3 \\
\cdot & \cdot & \cdot & -x_0 & \cdot & \cdot & x_1 & x_2 
\end{pmatrix}, 
d_3 = \begin{pmatrix} x_3 & -x_4 & \cdot \\
\cdot & x_3 & -x_4 \\
-x_2 & x_3 & \cdot \\
x_1 & -x_2 & \cdot \\
\cdot & -x_2 & x_3 \\
-x_0 & x_1 & \cdot \\
x_0 & \cdot & -x_2 \\
\cdot & -x_0 & x_1
\end{pmatrix}.
\]
In particular, the matrix $D_3 = d_2 \wedge d_3$ representing $3$rd linear syzygies (in the sense that it presents $K_{3,1} (S_X, S_1)$) among $Q_0, \cdots, Q_5$ is given by

\[
D_3 = 2 \begin{pmatrix} - x_2 \wedge x_3 &  x_2 \wedge x_4 & - x_3 \wedge x_4 \\
 x_1 \wedge x_3 & - (x_2 \wedge x_3 + x_1 \wedge x_4 ) &  x_2 \wedge x_4 \\
- x_0 \wedge x_3 &  (x_1 \wedge x_3 + x_0 \wedge x_4) & - x_1 \wedge x_4 \\
- x_1 \wedge x_2 &  x_1 \wedge x_3 & - x_2 \wedge x_3 \\
 x_0 \wedge x_2 & -(x_1 \wedge x_2 + x_0 \wedge x_3 ) &  x_1 \wedge x_3 \\
- x_0 \wedge x_1 &  x_0 \wedge x_2 & - x_1 \wedge x_2 \end{pmatrix}.
\]
For instance, the $1$st column of $D_3$ defines an element in $\wedge^2 S_1 \otimes (I_X)_2$ by tensor product with each of $Q_0, \cdots, Q_5$, namely, 
\[
\gamma = -x_2 \wedge x_3 \otimes Q_0 + x_1 \wedge x_3 \otimes Q_1 - x_0 \wedge x_3 \otimes Q_2 - x_1 \wedge x_2 \otimes Q_3 + x_0 \wedge x_2 \otimes Q_4 - x_0 \wedge x_1 \otimes Q_5 
\]
by clearing the common coefficient $2$. It is clear that the Koszul boundary map $\delta$ maps $\gamma$ to
\[
\delta(\gamma) = -x_3 \otimes x_2 Q_0 + x_2 \otimes x_3 Q_0 + \cdots - x_1 \otimes x_0 Q_5 + x_0 \otimes x_1 Q_5 = 0,
\]
hence, it defines a nonzero Koszul cycle in $K_{2,2} (I_X, S_1) \simeq K_{3,1} (S_X, S_1)$.

\item Let $Y \subset \p^4$ be the isomorphic projection of the Veronese surface $v_2 (\p^2) \subset \p^5$, defined as the $2$-minors of the symmetric matrix $M = \begin{pmatrix} x_0 & x_1 & x_2 \\ x_1 & x_3 & x_4 \\ x_2 & x_4 & x_5 \end{pmatrix}$, at the point $[1:0:0:1:0:1]$. The ideal of $Y$ is minimally generated by $7$ cubics
\[
\left\{
\begin{array}{rcl}
f_0 & = & y_0^3 - y_0 y_3^2 + y_0 y_2 y_4 + y_1 y_3 y_4 - y_0 y_4^2 \\
f_1 & = & y_0^2 y_1 - y_1 y_3^2 + y_0 y_3 y_4 \\
f_2 & = & y_0 y_1^2 + y_1 y_2 y_3 - y_0 y_3^2 \\
f_3 & = & y_1 ^3 - y_1 y_2^2 + y_0 y_2 y_3 - y_1 y_3^2 + y_1 y_2 y_4 \\
f_4 & = & y_0^2 y_2 - y_2 y_3^2 - y_1^2 y_4 + y_2^2 y_4 + y_3^2 y_4 - y_2 y_4^2 \\
f_5 & = & y_0 y_1 y_2 + y_1^2 y_3 - y_3^3 + y_2 y_3 y_4 \\
f_6 & = & y_0^2 y_3 - y_3^3 + y_0 y_1 y_4 + y_2 y_3 y_4.
\end{array}
\right.
\]

The minimal free resolution of $Y$ is consisted of linear syzygies of $f_0, \cdots, f_6$, namely, 
\[
0 \to S(-6) \stackrel{d_4} \to S(-5)^{\oplus 5} \stackrel{d_3} \to S(-4)^{\oplus 10} \stackrel{d_2} \to S(-3)^{\oplus 7} \stackrel{d_1} \to S \to 0
\]
where $d_1 = \begin{pmatrix} f_0 & \cdots & f_6 \end{pmatrix}$. For instance, a vector $\begin{pmatrix} 0 \\ y_1 \wedge y_2 \\ -y_0 \wedge y_2 + y_1 \wedge y_3 \\ -y_0 \wedge y_3 \\ 0 \\ y_0 \wedge y_1 - y_2 \wedge y_3 \\ y_2 \wedge y_3 \end{pmatrix}$ lies in the column space of $D_3 = d_2 \wedge d_3$, which corresponds to a syzygy $(y_1 \wedge y_2) \otimes f_1 + (-y_0 \wedge y_2 + y_1 \wedge y_3) \otimes f_2 - (y_0 \wedge y_3) \otimes f_3 + (y_0 \wedge y_1 - y_2 \wedge y_3 ) \otimes f_5 + (y_2 \wedge y_3) \otimes f_6$ in $K_{2,3} (I_X, S_1) \simeq K_{3,2} (S_X, S_1)$. 

Similarly, one can check that the matrix $D_4 = d_2 \wedge d_3 \wedge d_4$ is a single columned matrix $$6 \begin{pmatrix} - y_1 \wedge y_2 \wedge y_3 \\  (y_0 \wedge y_2 \wedge y_3 - y_1 \wedge y_2 \wedge y_4) \\ (y_0 \wedge y_2 \wedge y_4 - y_1 \wedge y_3 \wedge y_4) \\  y_0 \wedge y_3 \wedge y_4 \\ - y_0 \wedge y_1 \wedge y_3 \\ (-y_0 \wedge y_1 \wedge y_4 + y_2 \wedge y_3 \wedge y_4) \\ (y_0 \wedge y_1 \wedge y_2 - y_2 \wedge y_3 \wedge y_4) \end{pmatrix}, $$ and the induced element by tensor product with each of $f_0, \cdots, f_6$ in $\wedge^3 S_1 \otimes (I_X)_3$ gives a nonzero Koszul cycle in $K_{3,3} (I_X, S_1) \simeq k$.

\end{enumerate}
\end{exmp}

We now turn our focus on linear syzygies of smooth del Pezzo varieties. Note that a smooth del Pezzo variety $X \subseteq \p^r$ of dimension $n$, codimension $e$ is an arithmetically Gorenstein variety so that $\omega_X \simeq S_X (1-n)$. In this case, the minimal free resolution $F_{\bullet}$ of $X$ is 
\[
F_{\bullet} : 0 \to S(-e-2) \stackrel{d_e} \to S(-e)^{\beta_{e-1,1}} \stackrel{d_{e-1}} \to \cdots \stackrel{d_2} \to S(-2)^{\beta_{1,1}} \stackrel{d_1} \to S \to S_X \to 0
\]
where $\beta_{p,1} = p \binom{e+1}{p+1} - \binom{e}{p-1}$ \cite[Theorem 1]{Hoa93}. 

%

%
We introduce the main theorem of this paper which describes the structure of the highest linear syzygies of del Pezzo varieties.
\begin{thm}\label{thm:main}
Let $e \ge 3$. By taking suitable bases, the matrix $D_{e-1}$ representing the $(e-1)$-th linear syzygies among quadrics defining a smooth del Pezzo variety $X \subset \p^r$ of codimension $e$ becomes skew-symmetric.
\end{thm}

\begin{proof}
Thanks to the above discussion on a free resolution of arithmetically Gorenstein varieties, by taking suitable bases we may assume that
\[
D_{e-1} = d_2 \wedge d_3 \wedge \cdots \wedge d_m \wedge d_{m+1} \wedge d_m^{\vee} \wedge d_{m-1}^{\vee} \wedge \cdots \wedge d_2^{\vee}
\]
where $d_{m+1}^{\vee} =  (-1)^m d_{m+1}$ in the case $e=2m+1$ is odd, and
\[
D_{e-1} = d_2 \wedge d_3 \wedge \cdots \wedge d_m \wedge \varphi \wedge d_m^{\vee} \wedge \cdots \wedge d_2^{\vee}
\]
where $\varphi : F_m \otimes F_m \to F_e$ is the non-degenerate bilinear form which is either symmetric ($m$ is even) or skew-symmetric ($m$ is odd) in the case $e=2m$ is even. In both cases, the wedge product $D_{e-1}$ is either symmetric or skew-symmetric. 

Let us conclude that the latter one holds, that is, $D_{e-1}$ is always skew-symmetric. It is clear that the transpose $D_{e-1}^{\vee}$ is identical to
\[
d_2^{\vee \vee} \wedge d_3^{\vee \vee}  \wedge \cdots \wedge d_3^{\vee} \wedge d_2^{\vee} = d_2 \wedge d_3 \wedge \cdots \wedge d_3^{\vee} \wedge d_2^{\vee} 
\]
up to sign. Note that each entry of the matrix $D_e$ is a linear combination of wedge products of linear forms $\ell_1 \wedge \cdots \wedge \ell_{e-2}$, $\ell_i \in S_1$. The sign is determined from the comparison between two elements
\[
\ell_1 \wedge \ell_2 \wedge \cdots \wedge \ell_{e-2} \quad \text{ and } \ell_{e-2} \wedge \ell_{e-1} \wedge \cdots \wedge \ell_1.
\]
When $e=2m+1$ is odd, then these elements differ by $(-1)^{m-1}$, and hence $D_{e-1}^{\vee} = (-1)^{m+(m-1)} D_{e-1}$ since $d_{m+1}^{\vee} = (-1)^{m} d_{m+1}$. When $e=2m$ is even, then these elements also differ by $(-1)^{m-1}$, and hence $D_{e-1}^{\vee} = (-1)^{m+(m-1)} D_{e-1}$ since $\varphi^{\vee} = (-1)^m \varphi$. In any cases, we conclude that $D_{e-1}^{\vee} = - D_{e-1}$ as claimed. 
\end{proof}

Note that the above theorem also works for any arithmetically Gorenstein variety $X \subset \p^r$ of codimension $e \ge 3$ whose minimal free resolution $F_{\bullet}$ of $X$ is linear except both ends, that is, $F_i$ is generated in a single degree $(-d-i+1)$ for every $i=1, 2, \cdots, e-1$. In this manner, the case $e=3$ coincides with the famous Buchsbaum-Eisenbud structure theorem for arithmetically Gorenstein variety of codimension $3$ \cite[Theorem 2.1]{BE77}.

\begin{exmp} \label{exmp:skewSymmetricSyzygies}
\begin{enumerate}[(1)]
\item Let $X = \p^2 \times \p^2 \subset \p^8$ be the Segre fourfold, which is a del Pezzo variety of codimension $e=4$. Up to linear change of coordinates, $X$ is defined by $4 \times 4$ Pfaffians of the extrasymmetric matrix
\[
M = \left( \begin{array}{ccc|ccc}
\cdot & a & b & d & e & f \\
-a & \cdot & c& e & g & h \\
-b & -c & \cdot & f & h & i \\
\hline
-d & -e & -f &\cdot& -a & -b \\
-e & -g & -h & a & \cdot & -c \\
-f & -h & -i & b & c & \cdot
\end{array} \right)
\]
which is the case of the multiplier $\lambda = 1$, see \cite[Section 9]{BKR12} and \cite[Example 2.3]{Rei15}. Note that we intentionally follow the same notation as in \cite{Rei15} to avoid confusion. Among the fifteen Pfaffians, there are $9$ independent quadrics, and one can take the matrix $d_2 = \left( \begin{array}{cc} A&B \end{array} \right)$ of (linear) syzygies of $X$ as 

\[
\left( \begin{array}{rrrrrrrr|rrrrrrrr}
\cdot & -a & -b & -d & -e & -h & f & \cdot & i & \cdot&\cdot&\cdot&\cdot&\cdot&\cdot& c \\
a&\cdot&-c&-e&-g&\cdot&-h&f&\cdot&i&h&\cdot&\cdot&-b&-c&\cdot \\
b&c&\cdot&-f&-h&\cdot&\cdot&\cdot&\cdot&\cdot&i&\cdot&\cdot&\cdot&\cdot&\cdot \\
d&e&f&\cdot&a&c&-b&\cdot&\cdot&\cdot&\cdot&i&\cdot&\cdot&\cdot&-h \\
e&g&h&-a&\cdot&\cdot&c&-b&\cdot&\cdot&-c&\cdot&i&f&h&\cdot \\
\cdot&\cdot&\cdot&\cdot&\cdot&\cdot&-g&e&\cdot&h&\cdot&\cdot&c&-a&\cdot&\cdot \\
\cdot&\cdot&\cdot&\cdot&\cdot&g&\cdot&-d&-h&-f&e&-c&-b&\cdot&-a&\cdot \\
\cdot&\cdot&\cdot&\cdot&\cdot&-e&d&\cdot&f&\cdot&-d&b&\cdot&\cdot&\cdot&-a \\
\cdot&\cdot&\cdot&\cdot&\cdot&\cdot&\cdot&\cdot&-c&b&-a&-h&f&d&e&g
\end{array} \right)
\]
One can easily check that $d_3 = \left( \begin{array}{c} B^t \\ \hline A^t \end{array} \right)$ is a matrix of higher linear syzygies, and hence $D_3 = d_2 \wedge d_3$ is skew-symmetric.

\item Note that Theorem \ref{thm:main} is also applicable for a minimal free resolution of a Gorenstein ideal if it is linear except the both ends. As the simplest case of Gorenstein ideal of grade (= height) $e \ge 5$, we consider a complete intersection $I = (x_0, x_1, x_2, x_3, x_4)$ even though it does not appear as the ideal of a smooth non-degenerate projective variety in $\p^r$. Its syzygy matrix $d_2$ is consisted of Koszul relations
\[
d_2 = \begin{pmatrix} 
-x_4 & \cdot & \cdot & \cdot & -x_3 & \cdot & \cdot & -x_2 & \cdot & -x_1 \\
\cdot & -x_4 & \cdot & \cdot & \cdot & -x_3 & \cdot & \cdot & -x_2 & x_0 \\
\cdot & \cdot & -x_4 & \cdot &\cdot&\cdot&-x_3 & x_0 & x_1 & \cdot \\
\cdot & \cdot & \cdot & -x_4 & x_0 & x_1 & x_2 & \cdot &\cdot&\cdot \\
x_0 & x_1 & x_2 & x_3 & \cdot&\cdot&\cdot&\cdot&\cdot&\cdot
\end{pmatrix}.
\]
One can choose the matrix $d_3$ of syzygies of $d_2$ as a symmetric matrix, for instance, 
\[
d_3 = \begin{pmatrix} 
\cdot&\cdot&\cdot&\cdot&\cdot&-x_2&x_1&\cdot&x_3&\cdot \\
\cdot&\cdot&\cdot&\cdot&x_2&\cdot&-x_0&-x_3&\cdot&\cdot \\
\cdot&\cdot&\cdot&\cdot&-x_1&x_0&\cdot&\cdot&\cdot&x_3 \\
\cdot&\cdot&\cdot&\cdot&\cdot&\cdot&\cdot&x_1 & -x_0 & -x_2 \\
\cdot&x_2&-x_1&\cdot&\cdot&\cdot&\cdot&\cdot&-x_4&\cdot \\
-x_2&\cdot&x_0&\cdot&\cdot&\cdot&\cdot&x_4&\cdot&\cdot \\
x_1&-x_0&\cdot&\cdot&\cdot&\cdot&\cdot&\cdot&\cdot&-x_4 \\
\cdot&-x_3&\cdot&x_1&\cdot&x_4&\cdot&\cdot&\cdot&\cdot \\
x_3&\cdot&\cdot&-x_0&-x_4&\cdot&\cdot&\cdot&\cdot&\cdot \\
\cdot&\cdot&x_3&-x_2&\cdot&\cdot&-x_4&\cdot&\cdot&\cdot
\end{pmatrix}
\]
and the matrix $d_4$ of syzygies of $d_3$ as $d_4 = -d_2^{t}$. In particular, the matrix $D_4 = d_2 \wedge d_3 \wedge d_4$ representing $(e-1)$-th linear syzygies of $I$ is a skew-symmetric $5 \times 5$ matrix

\[
D_4 = 6 \begin{pmatrix}
\cdot & x_2 \wedge x_3 \wedge x_4 & -x_1 \wedge x_3 \wedge x_4 & x_1 \wedge x_2 \wedge x_4 & - x_1 \wedge x_2 \wedge x_3 \\
& \cdot & x_0 \wedge x_3 \wedge x_4 & -x_0 \wedge x_2 \wedge x_4 & x_0 \wedge x_2 \wedge x_3 \\
& & \cdot & x_0 \wedge x_1 \wedge x_4 & -x_0 \wedge x_1 \wedge x_3 \\
& & & \cdot & x_0 \wedge x_1 \wedge x_2 \\
& & & & \cdot
\end{pmatrix}.
\]
\end{enumerate}
\end{exmp}

\begin{rem}
Let $X \subset \p^r$ be a non-degenerate variety of codimension $e$. There are only a few possibilities that the quadrics defining $X$ admit $(e-1)$-th linear syzygies, in a connection with Green's $K_{p,1}$-theorem for the case $\beta_{e,1}(X) \neq 0$ \cite{Gre84, NP94, BS06, BS07}. Let us exclude the case when $X$ is a variety of minimal degree $\deg (X) = e+1$ which is equivalent to saying that $X$ has $e$-th linear syzygies. In this case, if we have a (skew-symmetric) $N \times N$ matrix $D_{e-1}$ with entries in $\wedge^{e-2} S_1$ where $N = \binom{e+1}{2} - 1$, and if $D_{e-1}$ is of full rank, then we have $\beta_{e-1,1}(X) \ge N = \binom{e+1}{2}-1$ which forces that $X$ is a del Pezzo variety \cite[Theorem 4.1, 4.3]{HK15}.  
\end{rem}

\section{Application to syzygy schemes}

One application of the above syzygy structure theorem is an ideal-theoretic description of quadric syzygy schemes. Roughly speaking, the \emph{(quadric) syzygy scheme} associated with a (higher) linear syzygy of $X$ is a scheme defined by quadrics that are minimally required to define the given syzygy. To be precise, the syzygy scheme of a Koszul cycle $\gamma \in K_{p,1}(S_X, S_1) \simeq K_{p-1,2} (I_X, S_1)$ is a quadratic scheme defined by an ideal $I(\gamma)$ which is generated by the quadrics $Q_1, \cdots, Q_s$ where $\gamma \in K_{p-1,2} (I_X, S_1)$ is minimally represented by 
\[
\gamma = \sum_{i=1}^s \ell_{J_i} \otimes Q_i \in \wedge^{p-1} S_1 \otimes (I_X)_2.
\]
This notion was first introduced by Green \cite{Gre82}, and then studied extensively by Schreyer and his students. In particular, thanks to the strong Castelnuovo lemma \cite[Theorem 3.c.6]{Gre84}, if $\deg (X) \ge e+3$ then the syzygy scheme of any nonzero Koszul cycle $\gamma \in K_{e-2,2}(I_X, S_1)$ must be a variety $Y$ of dimension $n+1$ and of minimal degree $\deg (Y) = e$. On the other hand, when $\deg (X) \le e+2$,  the syzygy scheme of $\gamma \in K_{e-2,2}(I_X, S_1)$ can differ from a variety of minimal degree. 

Using the above syzygy structure theorem, we immediately have the following ideal-theoretic description of a syzygy scheme of a highest (= $(e-1)$-th) linear syzygy of a smooth del Pezzo variety.
\begin{cor}
Let $X \subset \p^r$ be a smooth del Pezzo variety of codimension $e \ge 3$. The syzygy scheme $Syz(\gamma)$ of a nonzero Koszul cycle $\gamma \in K_{e-2,2}(I_X, S_1)$ is a scheme defined by $\binom{e}{2} \le m \le \binom{e+1}{2}-2$ quadrics in $I_X$. In particular, $Syz_{e-1} (X) = \bigcap_{\gamma \in K_{e-2,2}(I_X, S_1)} Syz(\gamma) = X$.
\end{cor}

\begin{proof}
Note first that $\dim_k (I_X)_2 = \binom{e+1}{2}-1$. The lower bound $\binom{e}{2} \le m$ is due to Green using the monodromy argument, see the proof of \cite[Theorem 1.3]{Gre82}. Since $\gamma$ is a nonzero $(e-1)$-th linear syzygy among quadrics $Q_1, \cdots, Q_N$ in $I_X$ (here, $N = \binom{e+1}{2}-1$ is the number of independent quadrics in $I_X$), it corresponds to a certain linear combination $\gamma = c_1 \gamma_1 + \cdots + c_N \gamma_N$ of syzygies $\gamma_1, \cdots, \gamma_N$ associated with the columns of the skew-symmetric matrix $D_{e-1}$ representing $(e-1)$-th syzygies of $Q_1, \cdots, Q_N$. In the case, one can choose a change of coordinate matrix $P$ whose first row is $\begin{pmatrix} c_1 & \cdots & c_N \end{pmatrix}$ and take $D^{\prime} = PD_{e-1} P^t$ so that $D^{\prime}$ is a matrix representing $(e-1)$th syzygies among quadrics $Q_1^{\prime}, \cdots, Q_N^{\prime}$ where $\begin{pmatrix} Q_1^{\prime} & \cdots & Q_N^{\prime} \end{pmatrix} P = \begin{pmatrix} Q_1 & \cdots & Q_N \end{pmatrix}$. The first column of $D^{\prime}$ corresponds to the syzygy  
\[
 D_{11}^{\prime} \otimes Q_1^{\prime} + \cdots + D_{N1}^{\prime} \otimes Q_N^{\prime} = \gamma,
\]
and in this expression we have $D_{11}^{\prime} = 0$ from the skew-symmetricity of $D^{\prime}$. Indeed, we do not need all of the $N$ quadrics to represent the syzygy $\gamma$ but at most $N-1$ quadrics $Q_2^{\prime}, \cdots, Q_N^{\prime}$ are enough. We conclude that $m \le \dim_k (I_X)_2 - 1 = \binom{e+1}{2} - 2$. 

Conversely, we can similarly construct a nonzero syzygy $\gamma \in K_{p-1,2} (I_X, S_1)$ whose syzygy ideal $I(\gamma)$ contains every nonzero quadric $0 \neq Q \in (I_X)_2$ using a suitable change of coordinates. In particular, we have $X \subseteq Syz_{e-1} (X) \subseteq X$ that completes the proof.
\end{proof}


\begin{exmp}
\begin{enumerate}[(1)]
\item Consider a smooth del Pezzo variety of the case $e=3$ and $n=6$, which is the Grassmannian $Gr(2,5) \subset \p^9$ of lines in $\p^4$. We may regard that it is defined by the $4$-Pfaffians $q_0, \cdots, q_4$ of the skew-symmetric matrix
\[
D_2 = d_2 = \begin{pmatrix} 
\cdot & x_{01} & x_{02} & x_{03} & x_{04} \\
& \cdot & x_{12} & x_{13} & x_{14} \\
& & \cdot & x_{23} & x_{24} \\
& & & \cdot & x_{34} \\
& & & & \cdot
\end{pmatrix}.
\]
For instance, the last column of $D_2$ corresponds to a linear syzygy $\gamma_4 = x_{04} \otimes q_0 + x_{14} \otimes q_1 + x_{24} \otimes q_2 + x_{34} \otimes q_3$, and hence the syzygy scheme $Syz (\gamma_4)$ is the scheme defined by the ideal $(q_0, q_1, q_2, q_3)$. One can easily verify that $Syz(\gamma_4) = V(q_0, q_1, q_2, q_3) \cup V(x_{04}, x_{14}, x_{24}, x_{34})$ is the (scheme-theoretic) union of $Gr(2,5)$ and a linear subspace defined by $4$ linear forms appearing in $\gamma_4$. It is clear that if we take the intersection of sufficiently many linear syzygies in $K_{1,2}( I_X, S_1)$ among $q_0, \cdots, q_5$, for instance, $Syz(\gamma_1) \cap Syz (\gamma_2) \cap Syz (\gamma_3) \cap Syz (\gamma_4)$ coincides with $Gr(2,5)$ itself. We refer to \cite[Section 6]{vB07} for a more detailed explanation in terms of the generic syzygy scheme and Grassmann syzygies.

\item In some cases of del Pezzo varieties, possibly there is an $(e-1)$-th syzygy $\gamma$ that is defined by less than $\binom{e+1}{2}-2$ quadrics. For example, let $X \subset \p^5$ be a del Pezzo surface of degree $5$ defined by $4$ Pfaffians $Q_0, \cdots, Q_4$ of the skew-symmetric matrix
\[
d_2 = \begin{pmatrix} 
\cdot & -x_0 + x_1 & -x_1 & x_1 - x_5 & x_5 \\
& \cdot & -x_2 & -x_5 & x_5 \\
& & \cdot & x_2 & -x_3 \\
& & & \cdot & x_4 \\
& & & & \cdot
\end{pmatrix}.
\]
The last column of the matrix $\gamma$ corresponds to the linear syzygy
\[
\gamma =  x_5 \otimes (Q_0 + Q_1) - x_3 \otimes Q_2 + x_4 \otimes Q_3.
\]
Hence, the syzygy scheme $Syz(\gamma)$ is defined by only three quadrics, namely, 
\[
\begin{array}{rl}
Q_0 + Q_1 = &  -x_1 x_3 + x_1 x_4 - x_2 x_4 \\
Q_2 = & -x_0 x_4 + x_1 x_4 - x_1 x_5 \\
Q_3 = & -x_0 x_3 + x_1 x_3 - x_1 x_5 + x_2 x_5.
\end{array}
\]
It is easy to check that the ideal $(Q_0+Q_1, Q_2, Q_3)$ is equal to the ideal generated by $2$-minors of $\begin{pmatrix} x_0 - x_1 & -x_1 + x_2 & x_1 \\ x_5 & x_3 & -x_4 \end{pmatrix}$, and thus $Syz(\gamma)$ is a rational cubic threefold scroll containing $X$. A similar discussion was made in \cite[Section 1]{KS89} using the notion of a generalized zero of a skew-symmetric matrix. We also refer to \cite[Remark 5.5]{KMP24} for an idea to find a rational normal scroll containing $X$ via quadric syzygy schemes.

\item Let us consider again the Segre variety in Example \ref{exmp:skewSymmetricSyzygies}-(1). We choose the first column of the skew-symmetric matrix $D_3$ representing $K_{3,1} (S_X, S_1) \simeq K_{2,2} (I_X, S_1)$, namely,
\[
\gamma = 2 \begin{pmatrix} 0 \\ -c \wedge f + b \wedge h + a \wedge i \\ b \wedge i \\ d \wedge i \\ -b \wedge c-f \wedge h+e \wedge i \\ -c \wedge e+a \wedge h \\ c \wedge d+b \wedge e-a \wedge f \\ -b \wedge d \\ a \wedge b+e \wedge f-d \wedge h
\end{pmatrix}.
\]
Hence, the syzygy scheme $Syz(\gamma)$ is defined by $8$ quadrics by dropping the $1$st quadric ($ac-fg+eh$) among the $9$ quadrics defining $X$ (we do not want to list all the quadric generators having $d_2$ as a matrix of linear syzygies among them, however, they can be easily computed from the syzygies of $d_2^t$). One can check that it is the union of $X=\p^2 \times \p^2$ and the linear subspace defined by the colon ideal $(I(\gamma) : I_X) = (a,b,c,d,e,f,h,i)$. Note that this linear subspace in $\p^8$ corresponds to the unique linear subspace $W \subset S_1$ of minimal dimension so that $\gamma \in K_{3,1} (S_X, S_1, W)$, see \cite[Lemma 3.1]{AN10} for more details on the uniqueness of such a linear subspace $W$. 

If we choose the last column of $D_3$ as a different example, that is, we take a different syzygy
\[
\gamma^{\prime} = 2 \begin{pmatrix} -a \wedge b-e \wedge f+d \wedge h \\ -a \wedge c+f \wedge g+e \wedge h \\ -b \wedge c+f \wedge h \\ c \wedge d-b \wedge e+a \wedge f \\ c \wedge e-b \wedge g+a \wedge h \\ e \wedge g \\ -d \wedge g \\ d \wedge e \\ 0 \end{pmatrix}
\]
whose syzygy scheme $Syz(\gamma^{\prime})$ is defined by $8$ quadrics dropping the last quadric ($cf-bh+ai$). One can also check that $Syz(\gamma^{\prime})$ is the union of $X$ and $V(a,b,c,d,e,f,g,h)$, since $\gamma^{\prime}$ is a Koszul cycle in $K_{3,1} (S_X, S_1, W^{\prime})$ where $W^{\prime} \subset S_1$ is the unique linear subspace of minimal dimension for defining $\gamma^{\prime}$.
\end{enumerate}
\end{exmp}

\begin{rem}
Let $X \subset \p^r$ be a smooth del Pezzo variety of codimension $e \ge 3$. When $\binom{e+1}{2}-2 \ge r+1$, one can computationally check that the unique linear subspace $W \subset S_1$ of minimal dimension to define a general highest linear syzygy $\gamma \in K_{e-2,2} (I_X, S_1)$ becomes $W=S_1$, see also \cite[Proposition 5.3]{KMP24}. In the case, the syzygy scheme $Syz(\gamma)$ is the same as $X$ itself even though the ideal $I(\gamma)$ is strictly smaller than $I_X$. Indeed, the colon ideal $(I(\gamma) : I_X)$ is the irrelavent maximal ideal in $S$ consisted of all the linear forms, and thus the saturation of $I(\gamma)$ is $I_X$.
\end{rem}


\begin{thebibliography}{Bib}

\bibitem[ALJ89]{ALJ89} 
B. Ang\'{e}niol and M. Lejeune-Jalabert, \emph{Calcul diff\'{e}rentiel et classes caract\'{e}ristiques en g\'{e}om\'{e}trie alg\'{e}brique}, Travaux en Cours \textbf{38}, Hermann, Paris (1989)

\bibitem[AN10]{AN10}
M. Aprodu and J. Nagel, \emph{Koszul cohomology and algebraic geometry}, Univ. Lecture Ser. \textbf{52}, American Mathematical Society, Providence, RI (2010)

\bibitem[vB07]{vB07}
H.-C. G. von Bothmer, \emph{Generic syzygy schemes}, J. Pure Appl. Algebra \textbf{208} (2007), 867--876

\bibitem[BS06]{BS06}
M. Brodmann and P. Schenzel, \emph{On varieties of almost minimal degree in small codimension}, J. Algebra \textbf{305} (2006), 789--801

\bibitem[BS07]{BS07}
M. Brodmann and P. Schenzel, \emph{Arithmetic properties of projective varieties of almost minimal degree}, J. Algebr. Geom. \textbf{16} (2007), 347--400

\bibitem[BKR12]{BKR12}
G. Brown, M. Kerber, and M. Reid, \emph{Fano $3$-folds in codimension $4$, Tom and Jerry. Part I}, Compositio Math. \textbf{148} (2012), 1171--1194

\bibitem[BE77]{BE77}
D. A. Buchsbaum and D. Eisenbud, \emph{Algebra structures for finite free resolutions, and some structure theorems for ideals of codimension $3$}, Amer. J. Math. \textbf{99} (1977), 447--485

\bibitem[ES03]{ES03}
D. Eisenbud and F.-O. Schreyer (with an appendix by J. Weyman), \emph{Resultants and Chow forms via exterior syzygies}, J. Amer. Math. Soc. \textbf{16} (2003), 537--579

\bibitem[Fuj90]{Fuj90}
T. Fujita, \emph{Classification theories of polarized varieties}, London Math. Soc. Lecture Note Ser. \textbf{155}, Cambridge Univ. Press, Cambridge (1990)

\bibitem[GKZ94]{GKZ94}
I. M. Gel'fand, M. M. Kapranov, and A. V. Zelevinsky, \emph{Discriminants, resultants, and multidimensional determinants}, Math. Theory Appl. Birkh{\"a}user Boston, Inc., Boston, MA (1994), 523pp

\bibitem[Gre82]{Gre82}
M. Green, \emph{The canonical ring of a variety of general type}, Duke Math. J. \textbf{49} (1982), 1087--1113

\bibitem[Gre84]{Gre84}
M. Green, \emph{Koszul cohomology and the geometry of projective varieties}, J. Diff. Geom. \textbf{19} (1984), 125--171

\bibitem[HK15]{HK15}
K. Han and S. Kwak, \emph{Sharp bounds for higher linear syzygies and classifications of projective varieties}, Math. Ann. \textbf{361} (2015), 535--561

\bibitem[Hoa93]{Hoa93}
Le Tuan Hoa, \emph{On minimal free resolutions of projective varieties of degree $=$ codimension $+2$}, J. Pure Appl. Algebra \textbf{87} (1993), 241--250

\bibitem[KMP24]{KMP24}
Y. Kim, H. Moon, and E. Park, \emph{Some remarks on the $\mathcal K_{p,1}$ theorem}, preprint available at arXiv:2404.03293.

\bibitem[KS89]{KS89}
J. Koh and M. Stillman, \emph{Linear syzygies and line bundles on an algebraic curve}, J. Algebra \textbf{125} (1989), 120--132

\bibitem[NP94]{NP94}
U. Nagel and Y. Pitteloud, \emph{On graded Betti numbers and geometrical properties of projective varieties}, Manuscripta Math. \textbf{84} (1994), 291--314

\bibitem[Rei15]{Rei15}
M. Reid, \emph{Gorenstein in codimension $4$: the general structure theory}, Algebraic geometry in East Asia -- Taipei 2011, Adv. Stud. Pure Math. \textbf{65} (2015), 201--227

\bibitem[Ste22]{Ste22}
I. Stenger, \emph{A structure result for Gorenstein algebras of odd codimension}, J. Algebra \textbf{589} (2022), 173--187
\end{thebibliography}
\def\cprime{$'$} \def\cprime{$'$} \def\cprime{$'$} \def\cprime{$'$}
  \def\cprime{$'$} \def\cprime{$'$} \def\dbar{\leavevmode\hbox to
  0pt{\hskip.2ex \accent"16\hss}d} \def\cprime{$'$} \def\cprime{$'$}
  \def\polhk#1{\setbox0=\hbox{#1}{\ooalign{\hidewidth
  \lower1.5ex\hbox{`}\hidewidth\crcr\unhbox0}}} \def\cprime{$'$}
  \def\cprime{$'$} \def\cprime{$'$} \def\cprime{$'$}
  \def\polhk#1{\setbox0=\hbox{#1}{\ooalign{\hidewidth
  \lower1.5ex\hbox{`}\hidewidth\crcr\unhbox0}}} \def\cdprime{$''$}
  \def\cprime{$'$} \def\cprime{$'$} \def\cprime{$'$} \def\cprime{$'$}
\providecommand{\bysame}{\leavevmode\hbox to3em{\hrulefill}\thinspace}
\providecommand{\MR}{\relax\ifhmode\unskip\space\fi MR }
\providecommand{\MRhref}[2]{%
  \href{http://www.ams.org/mathscinet-getitem?mr=#1}{#2}
}
\providecommand{\href}[2]{#2}

\vspace{0.5cm}

\end{document}